\documentclass[11pt]{article}
\usepackage{amsthm,amsfonts,amsmath,amscd,mathabx}
\swapnumbers
\theoremstyle{plain}

\newtheorem{thm}{THEOREM}[section]
\newtheorem{lm}[thm]{LEMMA}

\theoremstyle{definition}

\theoremstyle{definition}

\newcommand{\upchi}{\raise1pt\hbox{$\chi$}}

\newcommand{\Z}{{\mathord{\mathbb Z}}}
\newcommand{\N}{{\mathord{\mathbb N}}}

\renewcommand{\|}{{\Vert}}
\newcommand{\ii}{{\rm i}}
\newcommand{\e}{{\rm e}}

\numberwithin{equation}{section}
\begin{document}

\title{\bf On a conjecture by Hundertmark and Simon}

\author{\vspace{5pt} Ari Laptev$^{1}$,  Michael Loss$^{2}$, and Lukas Schimmer$^{3}$ \\
\vspace{5pt}\small{$1.$ Department of Mathematics, Imperial College London}\\[-6pt]
\small{London SW7 2AZ, UK} and SPBU\\
\vspace{5pt}\small{$2.$ School of Mathematics, Georgia Institute of
Technology,} \\[-6pt]
\small{Atlanta, GA 30332 USA}\\
\vspace{5pt}\small{$3.$ Institut Mittag--Leffler, The Royal Swedish Academy of Sciences} \\[-6pt]
\small{182 60 Djursholm, Sweden}\\
 }

\footnotetext [1]{Work of Michael Loss  is partially supported by U.S.
National Science Foundation
grant DMS 1856645.}
\footnotetext [2]{
Ari Laptev was partially supported by RSF grant 18-11-0032.}

\footnotetext{
\copyright\, 2020 by the authors. This paper may be reproduced, in
its entirety, for non-commercial purposes.}

\maketitle 

\textbf{Abstract}: The main result of this paper is a complete proof of a new Lieb-Thirring type inequality for Jacobi matrices
originally conjectured by Hundertmark and Simon. In particular  it is proved that the estimate on the sum of eigenvalues does not depend on the off-diagonal terms as long as they are smaller than their asymptotic value. An interesting feature of the proof is that it employs a technique originally used by Hundertmark-Laptev-Weidl concerning sums of singular values for compact operators. This technique seems to be novel in the context of Jacobi matrices.

\section{Introduction}
In this note we prove a conjecture of Hundertmark and Simon \cite{HS} concerning a sharp Lieb-Thirring inequality for Jacobi matrices.
We denote the symmetric Jacobi matrix with diagonal entries $\{ b_n\}_{n=-\infty}^\infty$ and off-diagonal entries $\{a_n\}_{n=-\infty}^\infty$ by
$$
J:= W(\{a_n\},\{b_n\})\ .
$$
It is assumed that the $a_n$ tend to $1$ as $n \to \pm \infty$ which yields the interval $[-2,2]$ as the essential spectrum of this Jacobi matrix.
We denote by $E_j^+(J)$ the eigenvalues of $J$ that are larger than $2$ and by $E_j^-(J)$ the eigenvalues of $J$ that are less than $-2$.
Hundertmark and Simon proved that
\begin{equation} \label{hundertsimon}
\sum_j  (E_j^+(J)^2-4)^{1/2}+(E_j^-(J)^2-4)^{1/2} \le \sum_n |b_n| + 4\sum_n |a_n-1| \ ,
\end{equation}
and observed that this inequality is sharp. Indeed, in the absence of the potential, they noted that the Jacobi matrix with all entries $a_n=1$ except for a single one that is chosen to be larger than one, yields equality in \eqref{hundertsimon}. They then conjectured the improved version of \eqref{hundertsimon} in which $|a_n-1|$ is replaced by $(a_n-1)_+$,
where we use the notation $(a)_+$ to mean $a$ if $a>0$ and $0$ if $a \le 0$. We have the following theorem.
\begin{thm}  \label{ourresult}
Assume that $\sum_n |b_n| < \infty$, $ \sum_n (a_n-1)_+ < \infty$ and $\lim_{n \to \pm \infty} a_n = 1$. Then we have the bound on the eigenvalue sum
\begin{equation}\label{main}
\sum_j  (E_j^+(J)^2-4)^{1/2}+(E_j^-(J)^2-4)^{1/2} \le \sum_n |b_n| + 4\sum_n (a_n-1)_+ \ .
\end{equation}
\end{thm}
The following consequence provides another justification for this short note. 
Generally, the proof of Lieb-Thirring inequalities is patterned after the ones for the continuum in which case the kinetic energy is given by $-\Delta$.  The discrete Laplacian requires that all $a_n=1$. It is, however, of some interest that in the case of Jacobi matrices this needs not be the case. This is a distinguishing feature of Lieb-Thirring inequalities for Jacobi matrices. Further, it was shown in
\cite{KS}, (see also \cite{S}) that
\begin{equation} \label{killipsimon}
\sum_j F(E_j^+(J)) + F(E_j^-(J)) \le \sum_n b_n^2 + 2 G(a_n)^2 ,
\end{equation}
where $G(a) = a^2-1-\log|a|^2$ and $F(E) = \beta^2 -\beta^{-2} - \log |\beta|^2$ and 
$E=\beta +\beta^{-1}$ with $|\beta|>1$.

In \cite{Sch} one of us proved that if  $a_n\equiv1$  then the inequality \eqref{hundertsimon} implies \eqref{killipsimon}. Using this argument and 
Theorem \ref{ourresult} we obtain as a consequence
\begin{thm}\label{anotherthm}
Let $\gamma>1/2$. Assume that $\sum_n b_n^{\gamma+1/2} < \infty$ and that $a_n\ge 0$ for all $n \in \Z$ with $\lim_{n \to \pm \infty} a_n = 1$ and $\sum_n (a_n-1)_+^{\gamma+1/2}<\infty$. Then with $\mathrm{B}(x,y)$ denboting the Beta function
\begin{equation} \label{anotherresult}
\sum_j \int_{2}^{|E_j^\pm(J)|}(t^2-4)^\frac12(|E_j^\pm(J)|-t)^{\gamma-\frac32} \,dt
\le \mathrm{B}(\gamma-1/2,2) \sum_n (\pm [b_n]_\pm\pm[a_n-1]_+\pm[a_{n-1}-1]_+)_\pm^{\gamma+\frac12} \ .
\end{equation}
\end{thm}

For $\gamma=3/2$ the left-hand side coincides with $\frac12\sum_{j}F(E_j^\pm(J))$. The function $G(a) \ge 0$ equals zero if and only if $a= \pm 1$ and hence \eqref{anotherresult} is an improvement over \eqref{killipsimon} for the case where $0 \le a_n \le 1$.

\medskip
\noindent
{\it Remark.}  
As proved in \cite{Sch}, the left-hand side in \eqref{anotherresult} is bounded from below by
$$
\sum_j  \int_2^{|E_j^\pm|}(t^2 - 4)^{\frac12}(|E_j^\pm|-t)^{\gamma-\frac32}\,dt
\ge 2\mathrm{B}(\gamma-1/2,3/2)\sum_{j}(|E_j^\pm|-2)^{\gamma}
$$
and by
$$
\sum_j  \int_2^{|E_j^\pm|}(t^2 - 4)^{\frac12}(|E_j^\pm|-t)^{\gamma-\frac32}\,dt
\ge \mathrm{B}(\gamma-1/2,2)\sum_{j}(|E_j^\pm|-2)^{\gamma+\frac12}\,.
$$
Thus \eqref{anotherresult} improves on corresponding Lieb-Thirring inequalities in \cite{HS}. Note that in \cite[p.121]{HS}  an argument is given that allows to replace $(a_n-1)$ in their results by $(a_n-1)_+$ for $\gamma\ge1$ but importantly not for $1/2<\gamma<1$ and not in the case of the main result \eqref{main}.

\medskip
In order to prove Theorem \ref{ourresult} we reduce the problem (as in \cite{HS})  to the discrete Schr\"odinger operator.  When treating terms $0 \le a_n \le 1$ we use some additional convexity property.

\section{The proof of the main result.}

We follow Hundertmark and Simon except for one key step. Using norm resolvent convergence we may  assume that only finitely many of the $a_n$ are not equal to one and finitely many $b_n$ are not equal to zero. Likewise, we may assume that $b_n \ge 0$. We also write
$$
W(\{a_n\},\{b_n\}) = A+B
$$
with the understanding that $A$ contains only the off-diagonal terms and $B$ the diagonal terms of the Jacobi matrix. If the off-diagonal  terms are all equal to one we denote the corresponding matrix by $A_1$.
The essential spectrum is given by the interval $[-2,2]$ which follows from Weyl's theorem. Let us denote the eigenvalues that are strictly greater than $2$ by
$E^+_1(A+B) \ge E^+_2(A+B)  \ge E^+_3(A+B)  \ge \cdots$. In what follows, the eigenvalues that are strictly less than $-2$ can be treated in a similar fashion.

\bigskip
\noindent
{\it Treating $a_n>1$:} 
Consider the window of the matrix $A$ that contains an off-diagonal term $a>1$ and use the elementary inequalities 
$$
\begin{pmatrix}
-a+1& 1\\
1 & -a+1
\end{pmatrix}
\le
\begin{pmatrix}
0 & a\\
a & 0 
\end{pmatrix} 
\le 
\begin{pmatrix}
a-1& 1\\
1 & a-1
\end{pmatrix}.
$$
Applying it to all $a_n>1$ we obtain 
$$
\widetilde J^- = W(\{\widetilde a_n\}, \{\widetilde b^-_n\} \le J = W(\{a_n\}, \{b_n\}) \le W(\{\widetilde a_n\}, \{\widetilde b^+_n\} = \widetilde J^+,
$$
where 
$$
\widetilde a_n =
\begin{cases}
 a_n, & {\rm if}  \,\,a_n\le1\\
 1, & {\rm if}  \,\,a_n >1
 \end{cases} 
 \qquad 
\widetilde b_n^\pm =
 \pm[b_n]_\pm  \pm[(a_{n-1}-1)_++  (a_n-1)_+] \ {\rm for. \ all } \ n \ ,
 $$
 where $[x]_\pm = \max\{\pm x,0\}$.
This implies
\begin{equation}\label{a_n>1}
\sum_j  (E_j^+(J)^2-4)^{1/2} \le \sum_j  (E_j^{+}(\widetilde J^+)^2-4)^{1/2}\ 
\end{equation}
and similarly
$$
\sum_j  (E_j^{-}(J)^2-4)^{1/2} \le \sum_j  (E_j^{-}(\widetilde J^-)^2-4)^{1/2}\ .
$$
This reduces the problem to the case $a_n\le1$, $n\in \Bbb Z$. 
\bigskip
\noindent

{\it Treating $a_n\le1$:}
Assuming $a_n\le 1$ we  consider the Birman-Schwinger operator
$$
K(A;\beta) := B^{1/2} (\beta - A)^{-1} B^{1/2},
$$
where $\beta>2$ and list the eigenvalues of the Birman-Schwinger operator, $E_j(B^{1/2} (\beta - A)^{-1} B^{1/2})$,  in decreasing order. The Birman-Schwinger principle states that the $j$th eigenvalue of $ B^{1/2} (E^+_j(A+B) - A)^{-1} B^{1/2}$ is one.

Let us decompose the matrix $A$ in a certain fashion. Consider the following window of the general matrix $A_\kappa$,
\begin{equation*}
\begin{pmatrix}
0 & a & 0 & 0 & 0\\
a & 0 & \kappa & 0 & 0\\
0 & \kappa & 0 & c & 0\\
0 & 0 & c & 0 & d\\
0 & 0 & 0 & d & 0
\end{pmatrix}.
\end{equation*}
The distinct notation $\kappa$ indicates that we concentrate on this particular position of the matrix. Denote by $U$ the infinite diagonal matrix that consists of $+1$ on the diagonal above the position of $\kappa$ and of $-1$ below $\kappa$, i.e., its window is given by
\begin{equation*} 
\begin{pmatrix}
1 & 0 & 0 & 0 & 0\\
0 & 1 & 0 & 0 & 0\\
0 & 0 & -1 & 0 & 0\\
0 & 0 & 0 & -1 & 0\\
0 & 0 & 0 & 0 & -1
\end{pmatrix}\ .
\end{equation*}
It has the effect that the corresponding window of the matrix $U \,A_\kappa \,U$ is given by
\begin{equation*} 
\begin{pmatrix}
0 & a & 0 & 0 & 0\\
a & 0 & -\kappa & 0 & 0\\
0 & -\kappa & 0 & c & 0\\
0 & 0 & c & 0 & d\\
0 & 0 & 0 & d & 0
\end{pmatrix} \, ,
\end{equation*}
i.e., the entry $\kappa$ changes sign and all others are unchanged. Since $U$ is unitary,
the matrices $A$ and $UA\,U$ are unitarily equivalent and have the same spectrum.
With a slight abuse of notation we now identify the matrices with their window. If we assume that $0\le \kappa <1$, we may write

\begin{equation*}
A_\kappa = 
\frac{\kappa+1}{2} \,
\begin{pmatrix}
0 & a & 0 & 0 & 0\\
a & 0 & 1 & 0 & 0\\
0 & 1 & 0 & c & 0\\
0 & 0 & c & 0 & d\\
0 & 0 & 0 & d & 0
\end{pmatrix} 
+   \frac{1-\kappa}{2}
\begin{pmatrix}
0 & a & 0 & 0 & 0\\
a & 0 & -1 & 0 & 0\\
0 & -1 & 0 & c & 0\\
0 & 0 & c & 0 & d\\
0 & 0 & 0 & d & 0
\end{pmatrix} \, ,
\end{equation*}
or if we denote by $A'$ the first matrix displayed above we can write
$$
A_\kappa =\frac{1+\kappa}{2} A' + \frac{1-\kappa}{2} UA'U \ .
$$
Repeating this for all the off-diagonal elements that are strictly less than one we find
\begin{equation} \label{convexcomb}
A = \sum_j \lambda_j U(j) \,  A_1 \, U(j),
\end{equation} 
where  off-diagonal elements of $A_1$ are equal one and  
where $ \lambda_j \ge 0, \, \sum_j \lambda_j = 1$. Since the matrices $U(j)$ are diagonal and have the matrix elements $\pm 1$,  the matrices $ A_1$ and $U(j) A_1 U(j)$ have the same eigenvalues.

\medskip
\noindent
The key observation is the following lemma

\begin{lm}\label{convex}
Let $\beta I > X$. Then the function
$$
X \rightarrow (\beta I - X)^{-1}
$$
is operator convex, i.e., if $0\le\lambda \le1$ then
$$
(\beta -\lambda X_1 - (1-\lambda) X_2)^{-1} 
\le \lambda (\beta - X_1)^{-1}+ (1 -\lambda) (\beta - X_2)^{-1} \ .
$$
\end{lm}
\begin{proof}
We follow \cite{HP}. Let $Y_j = \beta - X_j$, $j=1,2$.     It amounts to showing that for two positive and invertible self-adjoint operators $Y_1$ and $Y_2$ we have
$$
(\lambda Y_1 +(1-\lambda)Y_2)^{-1} \le \lambda Y_1^{-1} + (1-\lambda) Y_2^{-1} \ .
$$
This is equivalent to
$$
\left[Y_2^{1/2} \left(\lambda Y_2^{-1/2} Y_1 Y_2^{-1/2} +(1-\lambda)I\right) Y_2^{1/2}\right]^{-1} \le Y_2^{-1/2} \left[ \lambda Y_2^{1/2} Y_1^{-1} Y_2^{1/2} +(1-\lambda)I\right] Y_2^{-1/2}
$$
or
$$
Y_2^{-1/2} \left(\lambda Y_2^{-1/2} Y_1 Y_2^{-1/2} +(1-\lambda)I\right)^{-1} Y_2^{-1/2}\le Y_2^{-1/2} \left[ \lambda Y_2^{1/2} Y_1^{-1} Y_2^{1/2} +(1-\lambda)I\right] Y_2^{-1/2}
$$
which is equivalent to
$$
\left(\lambda Y_2^{-1/2} Y_1Y_2^{-1/2} +(1-\lambda)I\right)^{-1}\le  \lambda Y_2^{1/2} Y_1^{-1} Y_2^{1/2} +(1-\lambda) I \ .
$$
This is an inequality in terms of the positive, invertible and self-adjoint operator 
$Y= Y_2^{-1/2} Y_1Y_2^{-1/2}$, i.e.,
$$
(\lambda Y+(1-\lambda)I)^{-1} \le \lambda Y^{-1} +(1-\lambda)I,
$$
which reduces the whole problem to positive numbers on account of the spectral theorem. For positive numbers the inequality is obvious.
\end{proof}

Applying now Lemma \ref{convex} to \eqref{convexcomb}  we find
$$
K(A;\beta) \le \sum_j \lambda_j U(j) K(A_1; \beta) U(j) \ .
$$
If we set $\beta = \mu+ \frac{1}{\mu}$ and introduce the operator
$$
L_\mu(A) := (\beta^2-4)^{1/2} K(A;\beta)
$$
we find
$$
L_\mu(A) \le \sum_j \lambda_j U(j) L_\mu(A_1) U(j)  \ .
$$
The operator $ L_\mu(A_1) $ has the matrix representation
$$
 [L_\mu(A_1)]_{m,n} = b_m^{1/2} \mu^{|n-m|} b_n^{1/2}
 $$
where, once more
$$
\beta = \mu+ \frac{1}{\mu} \ , \mu < 1 \ .
$$
Denote by $S_n(\mu)$ the sum of the $n$ largest eigenvalues of $\sum_j \lambda_j U(j) L_\mu (A_1) U(j) $.
\begin{lm}\label{HSi}
$$
S_n(\mu) \le S_n(\nu)
$$
for $\nu \ge \mu$.
\end{lm}

We will present two proofs of this Lemma. 
\begin{proof}[Proof of Lemma \ref{HSi} following Hundertmark and Simon \cite{HS}]
Pick any bounded sequence $\{ \mu_n\}_{n=-\infty}^\infty$ and consider the matrix
$$
M_{\{\mu_n\}} := \sum_j \lambda_j U(j) L_{\{\mu_n\}} (A_1) U(j)\ ,
$$
where
$$
(L_{\{\mu_n\}})_{k,\ell} = \begin{cases} b_k^{1/2} \mu_k \cdots \mu_{\ell-1}  b_\ell^{1/2}  &{\rm if } \  k \le \ell  \\   (L_{\{\mu_n\}})_{\ell, k} & {\rm if} \ k > \ell   \ .\end{cases}
$$
Now we fix some integer $n$ and set $\mu_n = \mu$. All the other ones are fixed. The matrix $\sum_j \lambda_j U(j) L_{\{\mu_n\}}U(j) $ is an affine function of $\mu$
with a diagonal that is independent of $\mu$ and hence the matrix is of the form
$$
\left[\begin{array}{cc} A & 0 \\ 0 & B \end{array} \right] +\mu \left[\begin{array}{cc} 0 & C \\ C^* & 0 \end{array} \right] \ .
$$
Now we consider the sum of the top $n$ eigenvalues of this matrix and denote this function by $f(\mu)$. This function is convex. Moreover,
if we consider the diagonal matrix $V: \ell^2 \rightarrow \ell^2$ given by 
$$
[V\phi]_n = (-1)^n \phi_n 
$$
we find that
$$
V\sum_j \lambda_j U(j) L_{\{\mu_n\}}(A_1) U(j) V = \sum_j \lambda_j V U(j) L_{\{\mu_n\}}(A_1)U(j) V = \sum_j \lambda_j U(j)V L_{\{\mu_n\}}(A_1)VU(j) 
$$
since the matrix $V$ is also diagonal and commutes with the $U(j)$. This matrix has the same spectrum but is of the form
$$
\left[\begin{array}{cc} A & 0 \\ 0 & B \end{array} \right] -\mu \left[\begin{array}{cc} 0 & C \\ C^* & 0 \end{array} \right]
$$
and hence $f(\mu)$ is even. Thus $f$, being convex, is monotone for $\mu>0$. 
\end{proof}

\medskip
\noindent 
{\it Remark.} The above proof is closely related to the proof of a similar result from \cite{HLTh} except for using some symmetry property rather than perturbation at the spectral point zero.

\begin{proof}[Proof of Lemma \ref{HSi} following Hundertmark, Laptev and Weidl \cite{HLW}]
We aim to use the following abstract result of \cite{HLW} concerning the sum $\|T\|_n$ of the largest $n$ singular values of a compact operator $T$, 
$$
\|T\|_n=\sum_{j=1}^n\sqrt{E_j(T^*T)}\,.
$$
 The result is an immediate consequence of $\|T\|_n$ defining a norm by Ky-Fan's inequality. 
\begin{lm}\label{lem:maj} 
Let $T$ be a non-negative compact operator on a Hilbert space $\mathcal{G}$, let $g$ be a probability measure on $\Omega$, and let $\{V(k)\}_{k\in\Omega}$ be a family of unitary operators on $\mathcal{G}$. Then, for any $n\in\N$,
$$
\left\Vert\int_{\Omega}V(k)^*TV(k)\,dg(k)\right\Vert_n\le  \|T\|_n\,.
$$
\end{lm}

To apply the above result to $L_\mu(A_1)$, we recall the unitary map $\mathcal{F}:L^2([-\pi,\pi])\to \ell^2(\Z)$ onto the Fourier coefficients
$$(\mathcal{F}u)_n=\widehat{u}_n=\frac{1}{\sqrt{2\pi}}\int_{-\pi}^{\pi}\e^{\ii n k}u(k)d k\,,\qquad
(\mathcal{F}^*u)(k)=\widecheck{u}(k)=\frac{1}{\sqrt{2\pi}}\sum_{n} u_n\e^{-\ii n k}\,.
$$
By means of the transform $\mathcal{F}$, the free operator $W$ with $a_n\equiv1, b_n\equiv0$ is unitarily equivalent to the operator $-2\cos k$ on $L^2([-\pi,\pi])$. Defining $g_\mu$
to be the non-negative function
$$
g_\mu(k)=\frac{1}{\sqrt{2\pi}}\frac{\frac{1}{\mu}-\mu}{-2\cos k+\frac{1}{\mu}+\mu}
$$
and denoting by  $T$ the projection onto $b_n^{1/2}$, 
and by $V(k)$ the unitary operator  $(V(k)u)_n=\e^{-\ii n k} u_n$, we can thus write
$$
L_{\mu}(A_1)=\int_{-\pi}^\pi V(k)^*T V(k) \frac{{g}_{\mu}(k)}{\sqrt{2\pi}}d k\,.
$$
To obtain some more properties of $g_\mu$ we note that its Fourier transform is given by $(\widehat{g}_\mu)_n=\mu^{|n|}$. This can in particular be used to establish the aforementioned matrix representation of $L_\mu(A_1)$. For our purposes, we note that for $0<\mu<\nu< 1$  clearly
\begin{equation}
 (\widehat{g}_{\mu})_0=1\,,\qquad \widehat{g}_{\nu}\widehat{g}_{\mu/\nu}=\widehat{g}_{\mu}\,.
 \label{eq:propg}
\end{equation}
Since ${g}_{\mu}$ is smooth and periodic in $k$, the pointwise identity
$$
g_{\mu}(k)=\frac{1}{\sqrt{2\pi}}\sum_{n} (\widehat{g}_{\mu})_n\e^{-\ii n k}
$$
holds and the properties \eqref{eq:propg} imply
$$
\int_{-\pi}^\pi \frac{g_{\mu}(k)}{\sqrt{2\pi}}\, d k=1\,,
\qquad 
\frac{g_{\nu}}{\sqrt{2\pi}}* \frac{g_{\mu/\nu}}{\sqrt{2\pi}}
=\frac{g_{\mu}}{\sqrt{2\pi}}
$$
for all $0<\mu<\nu< 1$. 
The convolution identity is understood in the sense that
$$
\int_{-\pi}^\pi \frac{g_{\nu}(k-k')}{\sqrt{2\pi}}\frac{g_{\mu/\nu}(k')}{\sqrt{2\pi}}\,d k'=\int_{-\pi}^\pi \frac{g_{\nu}(k')}{\sqrt{2\pi}}\frac{g_{\mu/\nu}(k-k')}{\sqrt{2\pi}}\,d k'=\frac{g_{\mu}(k)}{\sqrt{2\pi}}
$$
which is well-defined since all three functions are periodic.  
Using this identity together with the fact that $V(k'+k'')=V(k')V(k'')$ and that $V(k')$ and $U(j)$ commute as both are multiplication operators, we obtain 
\begin{multline*}
\sum_{j}\lambda_j U(j)L_{\mu}(A_1)U(j)
=\sum_{j}\lambda_j\int_{-\pi}^\pi\int_{-\pi}^\pi U(j)V(k)^*T V(k)U(j) \frac{g_{\nu}(k-k')}{\sqrt{2\pi}}\frac{g_{\mu/\nu}(k')}{\sqrt{2\pi}}\,d  k'\,d  k\\
=\int_{-\pi}^\pi V(k')^*\left(\sum_{j}\lambda_j U(j)\int_{-\pi-k''}^{\pi-k''} V(k'')^*T V(k'') \frac{g_{\nu}(k'')}{\sqrt{2\pi}}\,d  k''U(j)\right)V(k')\frac{g_{\mu/\nu}(k')}{\sqrt{2\pi}}\,d  k'\,.
\end{multline*}
By periodicity, the operator 
$$
\int_{-\pi-k'}^{\pi-k'} V(k'')^*T V(k'') \frac{g_{\nu}(k'')}{\sqrt{2\pi}}\,d  k''
=\int_{-\pi}^{\pi} V(k'')^*T V(k'') \frac{g_{\nu}(k'')}{\sqrt{2\pi}}\,d  k''=L_{\nu}(A_1)
$$
is independent of $k'$  and  thus we can apply Lemma \ref{lem:maj} to obtain the desired monotonicity. The special case $\nu=1$ is an immediate consequence of Lemma \ref{lem:maj}.   
\end{proof}
\medskip
\noindent 
{\it Remark.} The above proof is closely related to the continuous case \cite{HLW} with the established convolution identity for $g_\mu$ replacing the fact that the Cauchy distribution is a convolution semigroup.  Using the parametrisation $\mu=\e^{-\sigma}, \nu=\e^{-\tau}$ with $\sigma>\tau>0$ the above identity may also be written in the more similar form $\frac{g_{\tau}}{\sqrt{2\pi}}*\frac{g_{\sigma-\tau}}{\sqrt{2\pi}}=\frac{g_{\sigma}}{\sqrt{2\pi}}$. 

\medskip
 
Because of the Birman-Schwinger principle, $E_j(K(A;E_j^+))$, the $j$-the eigenvalue of $K(A; E_j^+)$, equals $1$ and hence
$$
\sum_{j=1}^n (E_j^{+2}(J)-4)^{1/2}  = \sum_{j=1}^n  (E_j^{+2}(J)-4)^{1/2} E_j( K(A; E^+_j) ) = \sum_{j=1}^n E_j(L_{\mu_j}(A))
$$
$$
\le \sum_{j=1}^n  E_j(  \sum_k \lambda_kU(k) L_{\mu_j}(A_1) U(k)) \le  S_n(\mu_n)
$$
where $\mu_j+\frac{1}{\mu_j} = E^+_j$. Now, again, we proceed as in Hundertmark - Simon \cite{HS} and get the estimate
$$
S_{N^+}(\mu_n) \le {\rm Tr} \sum_j \lambda_j U(j) L_{\mu=1}(A_1) U(j)  = {\rm Tr} B \ ,
$$
where $S_{N^+}(\mu_n)$ includes all the eigenvalues of $\widetilde J$ that are greater than $2$. In other words
$$
 \sum_j  (E_j^{+2}(\widetilde J^+)-4)^{1/2} \le \sum_n [b_n]_+ + (a_{n-1}-1)_++  (a_n-1)_+ = \sum_n [b_n]_+ + 2 (a_n-1)_+ \ .
 $$

As shown in \cite{HS} the 
Jacobi matrices
$$
W(\{a_n\},\{b_n\}) \ {\rm and }  -W(\{a_n\}, \{-b_n\})
$$
are unitarily equivalent and hence it follows that
$$
 \sum_j  (E_j^{-2}(\widetilde J^-)-4)^{1/2} \le \sum_n [b_n]_- + (a_{n-1}-1)_++  (a_n-1)_+ = \sum_n [b_n]_-+ 2 (a_n-1)_+ \ ,
 $$
which together with the previous estimate proves Theorem  \ref{ourresult}.

As a corollary we obtain Theorem \ref{anotherthm}. 
\begin{proof}[Proof of Theorem \ref{anotherthm}]
Let $\chi_{\{c,d\}}$ be the characteristic function of the interval $(c,d)$. Then, recalling $J\le\widetilde{J}^+$, we have 
\begin{multline*}
\sum_j  \int_2^{E_j^+}(t^2 - 4)^{\frac12}(E_j^+-t)^{\gamma-\frac32} \, dt
=  \sum_j \int_0^\infty \left((E_j^+ (J-s) )^2 - 4\right)^{\frac12} s^{\gamma-\frac32}\chi_{_{\{E_j^+(J -s) \ge 2\}}} (s) \, ds\\
\le  \sum_j \int_0^\infty \left((E_j^+ (\widetilde{J}^+-s) )^2 - 4\right)^{\frac12} s^{\gamma-\frac32}\chi_{_{\{E_j^+(\widetilde{J}^+ -s) \ge 2\}}} (s) \, ds\,.
\end{multline*} 
Applying first the variational principle and then the main result we obtain 
$$
\sum_j  \int_2^{E_j^+}(t^2 - 4)^{\frac12}(E_j^+-t)^{\gamma-\frac32} \, dt
\le\sum_{n}\int_0^\infty(\widetilde{b}^+_n-s)_+s^{\gamma-\frac32}\,ds
= \mathrm{B}(\gamma-1/2,2)\, \sum_n (\widetilde{b}^+_n)_+^{\gamma+\frac12}
$$
and the proof is complete.

\end{proof}

{\bf Acknowledgment: } M.L. would like to thank Institute Mittag Leffler for its generous hospitality.

\end{document}